\documentclass[12pt,twoside]{amsart}

\usepackage{amsmath,amsthm,amscd,amssymb,mathrsfs,graphicx,amsfonts,mathrsfs}
\usepackage{amsfonts}
\usepackage{amssymb,enumerate}
\usepackage{amsthm}
\usepackage[all]{xy}
\usepackage{hyperref}
\allowdisplaybreaks[4] \footskip=12pt
\renewcommand{\uppercasenonmath}[1]{}

\numberwithin{equation}{section} \theoremstyle{plain}
\newtheorem*{thm*}{Main Theorem}
\newtheorem{thm}{Theorem}[section]
\newtheorem{cor}[thm]{Corollary}
\newtheorem*{cor*}{Corollary}
\newtheorem{lem}[thm]{Lemma}
\newtheorem*{lem*}{Lemma}

\newtheorem*{fact*}{Fact}

\newtheorem*{nota*}{Notation}
\newtheorem{prop}[thm]{Proposition}
\newtheorem*{prop*}{Proposition}
\newtheorem{rem}[thm]{Remark}
\newtheorem*{rem*}{Remark}

\newtheorem*{observation*}{Observation}
\newtheorem{exa}[thm]{Example}
\newtheorem*{exa*}{Example}
\newtheorem{df}[thm]{Definition}
\newtheorem*{df*}{Definition}

\newtheorem*{con*}{Construction}

\renewcommand{\geq}{\geqslant}
\renewcommand{\leq}{\leqslant}

\begin{document}
\begin{center}
{\large  \bf Another version of cosupport for complexes}

\vspace{0.5cm}  Xiaoyan Yang\\
%\bigskip
Department of Mathematics, Northwest Normal University, Lanzhou 730070,
China
E-mail: yangxy@nwnu.edu.cn
\end{center}

\bigskip
\centerline { \bf  Abstract}
%\bigskip
\leftskip10truemm \rightskip10truemm \noindent
The goal of the article is to develop a theory dual to that of support in the derived category $\mathrm{D}(R)$, this is done by
introducing another versions of the ``big'' and ``small'' cosupport for complexes that are differ from the cosupport in [J. Reine Angew. Math. 673 (2012) 161--207]. We provide some properties for cosupport that are similar--or rather dual--to those of support for complexes, study some relations between the ``big'' and ``small'' cosupport and give some computations and comparisons of the ``small'' support and cosupport. Finally, we investigate the dual notion of associated primes of complexes.\\
\vbox to 0.3cm{}

\leftskip0truemm \rightskip0truemm
\bigskip
\section* { \bf Introduction}
%\bigskip
Support is a fundamental concept in commutative algebras, which provides a geometric approach for
studying various algebraic structures. Based on certain localization functors on compactly generated triangulated categories, Benson, Iyengar and Krause \cite{BIK,BIK1,BIK2} developed the theories of support and cosupport. Suitably specialized their approach recovers the support theory of Foxby \cite{F}
and Neeman \cite{N1,N} for commutative noetherian rings, the theory of Avramov and Buchweitz
for complete intersection local rings \cite{A,AB}. Their works also play a pivotal role on a classification theorem for the thick subcategories of modules and the localizing subcategories of the stable module category (see \cite{HK} and \cite{BIK4,BIK3}).

 Despite the many ways in which cosupport is dual to the notion of support, cosupport seems to be more mysterious, even in the setting of a commutative noetherian ring. In general the theory of cosupport is not completely satisfactory because this construction is not as
well understood as support. Richardson \cite{R} introduced the concept of cosupport of modules and proved that the cosupport have properties dual to those of support.

One purpose of this
paper is to extend the concept of cosupport in \cite{R} to unbounded complexes.
We focus on the duality functor $D_R(-)=\mathrm{Hom}_R(-,\bigoplus E(R/\mathfrak{m}))$ the sum running over all maximal ideals $\mathfrak{m}$ of $R$, where $E(R/\mathfrak{m})$ is the injective envelope
 of $R/\mathfrak{m}$. For an $R$-complex $M$, the co-localization of $M$ relative
to a prime ideal $\mathfrak{p}$ is the $R_\mathfrak{p}$-complex\begin{center}${^\mathfrak{p}}M=\mathrm{Hom}_{R_\mathfrak{p}}(D_R(M)_\mathfrak{p},
E_{R_\mathfrak{p}}(k(\mathfrak{p})))
\simeq\mathrm{Hom}_{R}(D_R(M),
E_{R}(R/\mathfrak{p}))$.\end{center}
In Section 2, we define the set $\mathrm{coSupp}_RM$ of ``big'' cosupport of $M$ to be the set of prime ideals $\mathfrak{p}$ such that ${^\mathfrak{p}}M\not\simeq0$.
 One of the main results of this work is that $\mathrm{coSupp}_RM$ can be detected by the big cosupport of the homology of $M$. We show that

 \vspace{2mm} \noindent{\bf Theorem A.}\label{Th1.4} {\it{For any $R$-complex $M$, one has that \begin{center}$\mathrm{coSupp}_RM=\bigcup_{i\in\mathbb{Z}}\mathrm{coSupp}_R\mathrm{H}_i(M)$.\end{center}In particular, $M\not\simeq0$ if and only if $\mathrm{coSupp}_RM\neq\emptyset$.}}

 We provide the following (partial) duality between the big cosupport and support.

 \vspace{2mm} \noindent{\bf Theorem B.}\label{Th1.4} {\it{Let $M$ be an $R$-complex.

$\mathrm{(1)}$ $\mathfrak{p}\in\mathrm{coSupp}_RM$ if and only if $\mathfrak{p}\in\mathrm{Supp}_RD_R(M)$.

$\mathrm{(2)}$ If $\mathfrak{p}\in\mathrm{Supp}_RM$, then $\mathfrak{p}\in\mathrm{coSupp}_RD_R(M)$. The converse holds when $M\in\mathrm{D}^\mathrm{n}(R)$ $($i.e. each $\mathrm{H}_i(M)$ is noetherian$)$.}}

By an example we show that the above notion is not the same as the one in \cite{WW}.

Section 3 investigates the ``small'' cosupport of complexes \begin{center}$\mathrm{cosupp}_RM:=\{\mathfrak{p}\in\mathrm{Spec}R\hspace{0.03cm}|\hspace{0.03cm}
\mathrm{RHom}_R(R/\mathfrak{p},{^\mathfrak{p}}M)\not\simeq0\}$,\end{center} and proves some properties for ``small'' cosupport that are similar to those of ``small'' support and ``big'' cosupport in Section 2.

In Section 4, we study some relations between the ``big'' and ``small'' cosupport, and show $\mathrm{cosupp}_RM\subseteq\mathrm{coSupp}_RM$. By an example we show that the inclusion may be strict.

 Section 5 is devoted to provide some computations of the ``small'' support and ``small'' cosupport, and study the relation between $\mathrm{cosupp}_RM$ and $\mathrm{cosupp}_R\mathrm{H}(M)$. As an application, we give the comparison of the support and cosupport.

 The concept of coassociated primes of complexes is introduced in the last section, and an extension of Nakayama lemma is given.

\bigskip
\section{\bf Preliminaries}
Unless stated to the contrary we assume throughout this paper that $R$ is a commutative noetherian ring which is not necessarily local.

This section is devoted to recalling some notions and basic consequences
for use throughout this paper. For terminology we shall follow \cite{CF} and \cite{WW}.

\vspace{2mm}
{\bf Complexes.} The category of chain $R$-complexes is denoted by $\mathrm{C}(R)$.
The derived category of $R$-complexes is denoted by $\mathrm{D}(R)$.

Let $M$ be an object in $\mathrm{C}(R)$ and $n\in\mathbb{Z}$.
The soft right-truncation, $\sigma_{\geq n}(M)$, of $M$ at $n$ and the soft left-truncation, $\sigma_{\leq n}(M)$, of $M$ at $n$ are given by
\begin{center}$\sigma_{\geq n}(M):\ \cdots\longrightarrow M_{n+2}\stackrel{d_{n+2}}\longrightarrow M_{n+1}\stackrel{d_{n+1}}\longrightarrow
\textrm{Ker}d_n\longrightarrow 0$,\end{center}
\begin{center}$\sigma_{\leq n}(M):\ 0\longrightarrow \textrm{Coker}d_{n+1}\stackrel{\overline{d}_{n}}\longrightarrow M_{n-1}\stackrel{d_{n-1}}\longrightarrow
M_{n-2}\longrightarrow\cdots$.\end{center}The differential $\overline{d}_{n}$ is the induced morphism on residue classes.

An $R$-complex $M$ is called bounded
above if $\mathrm{H}_n(M)=0$ for all $n\gg0$, bounded below if $\mathrm{H}_n(M)=0$ for all $n\ll0$, and
bounded if it both bounded above and bounded below. The full triangulated subcategories
consisting of bounded above, bounded below and bounded $R$-complexes are denoted by
$\mathrm{D}_-(R),\mathrm{D}_+(R)$ and $\mathrm{D}_\mathrm{b}(R)$.
We denote by $\mathrm{D}^\mathrm{n}(R)$ the full triangulated subcategory of $\mathrm{D}(R)$ consisting of $R$-complexes $M$ such that $\mathrm{H}_i(M)$ are
noetherian $R$-modules for all $i$, and denote by $\mathrm{D}^\mathrm{a}(R)$ the full triangulated subcategory of $\mathrm{D}(R)$ consisting of $R$-complexes $M$ such that $\mathrm{H}_i(M)$ are
artinian $R$-modules for all $i$.
For $M\in\mathrm{D}(R)$,
\begin{center}$\mathrm{inf}M:=\mathrm{inf}\{n\in\mathbb{Z}\hspace{0.03cm}|\hspace{0.03cm}\mathrm{H}_n(M)\neq0\},\quad \mathrm{sup}M:=\mathrm{sup}\{n\in\mathbb{Z}\hspace{0.03cm}|\hspace{0.03cm}\mathrm{H}_n(M)\neq0\}$.\end{center}

We write $\mathrm{Spec}R$ for the set of
prime ideals of $R$ and $\mathrm{Max}R$ for the set of
maximal ideals of $R$. For an ideal $\mathfrak{a}$ in $R$ and $\mathfrak{p}\in\mathrm{Spec}R$, we set \begin{center}$\mathrm{U}(\mathfrak{p})=\{\mathfrak{q}\in \mathrm{Spec}R\hspace{0.03cm}|\hspace{0.03cm} \mathfrak{q}\subseteq\mathfrak{p}\}$ and $\mathrm{V}(\mathfrak{a})=\{\mathfrak{q}\in\textrm{Spec}R\hspace{0.03cm}|\hspace{0.03cm}
\mathfrak{a}\subseteq\mathfrak{q}\}$.\end{center}
Denote $D_R(-)=\mathrm{Hom}_R(-,\bigoplus_{\mathfrak{m}\in\mathrm{Max}R}E(R/\mathfrak{m}))$ and $D_\mathfrak{m}(-)=\mathrm{Hom}_R(-,E(R/\mathfrak{m}))$ for $\mathfrak{m}\in\mathrm{Max}R$. Let $S$ be a multiplicatively closed subset of $R$. For an $R$-complex $M$, the co-localization of $M$ relative
to $S$ is the $S^{-1}R$-module $S_{-1}M=D_{S^{-1}R}(S^{-1}D_R(M))$. If $S=R-\mathfrak{p}$
for some $\mathfrak{p}\in\mathrm{Spec}R$, we write ${^\mathfrak{p}}M$ for $S_{-1}M$.
We also set $M^\thicksim=\prod_{\mathfrak{m}\in\mathrm{Max}R}D_\mathfrak{m}(D_\mathfrak{m}(M))$.

\vspace{2mm}
{\bf Support and cosupport.} The ``small'' support of an $R$-complex $M$ is the set \begin{center}$\mathrm{supp}_RM=\{\mathfrak{p}\in\mathrm{Spec}R\hspace{0.03cm}|\hspace{0.03cm}k(\mathfrak{p})\otimes^\mathrm{L}_RM\not\simeq0\}$,\end{center}where $k(\mathfrak{p})=R_\mathfrak{p}/\mathfrak{p}R_\mathfrak{p}$.  The ``big'' support of $M$ is the set \begin{center}$\mathrm{Supp}_RM=\{\mathfrak{p}\in\mathrm{Spec}R\hspace{0.03cm}|\hspace{0.03cm}
M_\mathfrak{p}\not\simeq0\}$.\end{center}It follows from \cite[6.4.2.1, 6.1.3.2]{CF} that $\mathrm{supp}_RM\subseteq\mathrm{Supp}_RM=\bigcup_{i\in\mathbb{Z}}\mathrm{Supp}_R\mathrm{H}_i(M)$.

The ``small'' cosupport of an $R$-complex $M$ is the set \begin{center}$\mathrm{co\textrm{-}supp}_RM=\{\mathfrak{p}\in\mathrm{Spec}R\hspace{0.03cm}|\hspace{0.03cm}
\mathrm{RHom}_R(k(\mathfrak{p}),M)\not\simeq0\}$.\end{center}  The ``big'' cosupport of $M$ is the set \begin{center}$\mathrm{Co\textrm{-}supp}_RM=\{\mathfrak{p}\in\mathrm{Spec}R\hspace{0.03cm}|\hspace{0.03cm}
\mathrm{RHom}_R(R_\mathfrak{p},M)\not\simeq0\}$.\end{center}It follows from \cite[Corollary 4.6]{WW} that $\mathrm{co\textrm{-}supp}_RM\subseteq\mathrm{Co\textrm{-}supp}_RM$.

Richardson \cite{R} defined the cosupport of an $R$-module $K$, $\mathrm{coSupp}_RK$, as the set \begin{center}$\mathrm{coSupp}_RK:=\{\mathfrak{p}\in\mathrm{Spec}R\hspace{0.03cm}|\hspace{0.03cm}
{^\mathfrak{p}}K\neq0\}$.\end{center}Yassemi \cite{Y} introduced the cocyclic modules and another cosupport of modules. An $R$-module $L$ is cocyclic if $L$ is a submodule of $E(R/\mathfrak{m})$ for some $\mathfrak{m}\in\mathrm{Max}R$. The cosupport of $K$ is defined as the set
of prime ideals $\mathfrak{p}$ such that there is a cocyclic homomorphic image $L$
of $K$ such that $\mathfrak{p}\subseteq\mathrm{Ann}_RL$, the annihilator of $L$, and denoted this set by $\mathrm{Cosupp}_RK$.

\begin{lem}\label{lem:0.0}{\it{Let $K$ be an $R$-module and $\mathfrak{p}$ a point in $\mathrm{Spec}R$.

$\mathrm{(1)}$ If $\mathfrak{p}\in\mathrm{Cosupp}_RK$, then $\mathfrak{p}\in\mathrm{coSupp}_RK$.

$\mathrm{(2)}$ If $R$ is a semi-local ring or $K$ is a finitely generated $R$-module, then $\mathfrak{p}\in\mathrm{coSupp}_RK$ if and only if $\mathfrak{p}\in\mathrm{Cosupp}_RK$.}}
\end{lem}
\begin{proof} (1) The exact sequence $0\rightarrow E(R/\mathfrak{m})\rightarrow\bigoplus_{\mathfrak{m}\in\mathrm{Max}R}E(R/\mathfrak{m})\rightarrow
\bigoplus_{\mathfrak{m}\neq\mathfrak{m'}\in\mathrm{Max}R}E(R/\mathfrak{m'})\rightarrow0$ induces the following short exact sequence \begin{center}$0\rightarrow \mathrm{Hom}_R(\mathrm{Hom}_R(K,\bigoplus_{\mathfrak{m}\neq\mathfrak{m'}\in\mathrm{Max}R}E(R/\mathfrak{m'})),
E(R/\mathfrak{p}))\rightarrow\mathrm{Hom}_R(D_R(K),
E(R/\mathfrak{p}))\rightarrow
\mathrm{Hom}_R(D_\mathfrak{m}(K),E(R/\mathfrak{p}))\rightarrow0$.\end{center}By the remark after \cite[Theorem 3.8]{Y}, one has $\mathrm{Hom}_{R}(\bigoplus_{\mathfrak{m}\in\mathrm{Max}R}D_\mathfrak{m}(K),
E(R/\mathfrak{p}))\neq0$, then $\mathrm{Hom}_{R}(D_\mathfrak{m}(K),E(R/\mathfrak{p}))\neq0$ for some $\mathfrak{m}$, and hence ${^\mathfrak{p}}K\neq0$ and $\mathfrak{p}\in\mathrm{coSupp}_RK$.

(2) If $R$ is semi-local or $K$ is finitely generated, then ${^\mathfrak{p}}K\cong\mathrm{Hom}_{R}(\bigoplus_{\mathfrak{m}\in\mathrm{Max}R}D_\mathfrak{m}(K),
E(R/\mathfrak{p}))$. Hence the equivalence follows from the remark after \cite[Theorem 3.8]{Y}.
\end{proof}

\begin{lem}\label{lem:0.1}{\it{Let $K$ be an $R$-module and $\mathfrak{p}$ a point in $\mathrm{Spec}R$. If $\mathfrak{p}\in\mathrm{Supp}_RK$, then $\mathfrak{p}\in\mathrm{coSupp}_RD_R(K)$. The converse holds when $K$ is finitely generated.}}
\end{lem}
\begin{proof} Since $\mathfrak{p}\in\mathrm{Supp}_RK$, $\mathfrak{p}\in\mathrm{Cosupp}_RD_\mathfrak{m}(K)$ for some $\mathfrak{m}\in\mathrm{Max}R\cap\mathrm{V}(\mathfrak{p})$ by \cite[Lemma 2.8]{Y}, and hence $\mathfrak{p}\in\mathrm{Cosupp}_RD_R(K)$. Consequently, $\mathfrak{p}\in\mathrm{coSupp}_RD_R(K)$ by Lemma \ref{lem:0.0}. Conversely, if $\mathfrak{p}\in\mathrm{coSupp}_RD_R(K)$ then ${^\mathfrak{p}}D_R(K)\neq0$, which implies that $K_\mathfrak{p}\neq0$ since $K$ is finitely generated. Therefore, $\mathfrak{p}\in\mathrm{Supp}_RK$.
\end{proof}

\bigskip
\section{\bf Another version of big cosupport}
This section introduces the set $\mathrm{coSupp}_RM$ of the ``big'' cosupport of an $R$-complex $M$, which is differ from the ``big'' cosupport in \cite{WW}. We show that $\mathrm{coSupp}_RM$ is completely related to $\mathrm{coSupp}_R\mathrm{H}_i(M)$, and give a (partial) duality between $\mathrm{coSupp}_RM$ and $\mathrm{Supp}_RM$.

\begin{df}\label{lem:1.1}{\rm Let $M$ be an $R$-complex.
The ``big'' cosupport of $M$ is defined as
\begin{center}$\mathrm{coSupp}_RM:=\{\mathfrak{p}\in\mathrm{Spec}R\hspace{0.03cm}|\hspace{0.03cm}
{^\mathfrak{p}}M\not\simeq0\}$.\end{center}}
\end{df}

The next theorem establishes the fact that the big cosupport for an $R$-complex is completely related to the big cosupport of the homology modules of complexes, which bring an analogue of the big support (see \cite[6.1.3.2]{CF}).

\begin{thm}\label{lem:1.2}{\it{Let $M$ be an $R$-complex. One has an equality \begin{center}$\mathrm{coSupp}_RM=\bigcup_{i\in\mathbb{Z}}\mathrm{coSupp}_R\mathrm{H}_i(M)$.\end{center}}}
\end{thm}
\begin{proof} One has the following equivalences\begin{center}$\begin{aligned}\mathfrak{p}\in\mathrm{coSupp}_RM
&\Longleftrightarrow\mathrm{H}_i({^\mathfrak{p}}M)
\neq0\ \textrm{for\ some}\ i\\
&\Longleftrightarrow\mathrm{Hom}_R(D_R(\mathrm{H}_i(M)),E(R/\mathfrak{p}))
\neq0\ \textrm{for\ some}\ i\\
&\Longleftrightarrow{^\mathfrak{p}}\mathrm{H}_{i}(M)\neq0\ \textrm{for\ some}\ i\\
&\Longleftrightarrow\mathfrak{p}\in\bigcup_{i\in\mathbb{Z}}\mathrm{coSupp}_R\mathrm{H}_i(M),\end{aligned}$\end{center}
where the second equivalence is because $E(R/\mathfrak{m})$ and $E(R/\mathfrak{p})$ are injective.
\end{proof}

\begin{cor}\label{lem:1.4}{\it{For an $R$-complex $M$, one has $M\not\simeq0$ if and only if $\mathrm{coSupp}_RM\neq\emptyset$.}}
\end{cor}
\begin{proof} One has that $\mathrm{coSupp}_RM\neq\emptyset$ if and only if $\mathrm{coSupp}_R\mathrm{H}_i(M)\neq\emptyset$ for some $i$ if and only if $\mathrm{H}_i(M)\neq0$ for some $i$ if and only if $M\not\simeq0$, where the first equivalence is by Theorem \ref{lem:1.2}, the second one is by \cite[Theorem 2.7]{R}.
\end{proof}

If $0\not\simeq M\in\mathrm{D}^\mathrm{n}_\mathrm{b}(R)$, then $\mathrm{Supp}_RM=\mathrm{V}(\mathrm{Ann}_RM)$.
The next corollary is dual to this.

\begin{cor}\label{lem:1.00}{\it{For any $0\not\simeq M\in\mathrm{D}^\mathrm{a}_\mathrm{b}(R)$, one has that \begin{center}$\mathrm{coSupp}_RM=\mathrm{V}(\mathrm{Ann}_RM)=\mathrm{Supp}_R(R/\mathrm{Ann}_RM)$.\end{center}}}
\end{cor}
\begin{proof} Set $i=\mathrm{inf}M$ and $s=\mathrm{sup}M$. We have \begin{center}$\begin{aligned}\mathrm{coSupp}_{R}M
&=\bigcup_{j=i}^s\mathrm{coSupp}_{R}\mathrm{H}_{j}(M)\\
&=\bigcup_{j=i}^s\mathrm{V}(\mathrm{Ann}_R\mathrm{H}_{j}(M))\\
&=\mathrm{V}(\bigcap_{j=i}^s\mathrm{Ann}_R\mathrm{H}_{j}(M))\\
&=\mathrm{V}(\mathrm{Ann}_RM),\end{aligned}$\end{center}
where the second equality is by \cite[Theorem 2.7]{R}.
\end{proof}

The following result play an important role in the rest of the paper.

\begin{thm}\label{lem:1.3}{\it{Let $M$ be an $R$-complex. The following are equivalent:

$\mathrm{(1)}$ $\mathfrak{p}\in\mathrm{coSupp}_RM$;

$\mathrm{(2)}$ $\mathfrak{p}\in\mathrm{Supp}_RD_R(M)$.\\
If in addition $R$ is semi-local, then $(1)$ and $(2)$ are equivalent to

$\mathrm{(3)}$ $\mathfrak{p}\in\mathrm{Supp}_RD_\mathfrak{m}(M)$ for some $\mathfrak{m}\in\mathrm{Max}R\cap\mathrm{V}(\mathfrak{p})$;

$\mathrm{(4)}$ $\mathrm{RHom}_R(R_\mathfrak{p},M^\thicksim)\not\simeq0$.}}
\end{thm}
\begin{proof} (1) $\Leftrightarrow$ (2) One has the following equivalences\begin{center}$\begin{aligned}\mathfrak{p}\in\mathrm{coSupp}_RM
&\Longleftrightarrow\mathfrak{p}\in\mathrm{coSupp}_R\mathrm{H}_i(M)\ \textrm{for\ some}\ i\\
&\Longleftrightarrow\mathfrak{p}\in\mathrm{Supp}_RD_R(\mathrm{H}_i(M))\ \textrm{for\ some}\ i\\
&\Longleftrightarrow\mathfrak{p}\in\mathrm{Supp}_R\mathrm{H}_{-i}(D_R(M))\ \textrm{for\ some}\ i\\
&\Longleftrightarrow\mathfrak{p}\in\mathrm{Supp}_RD_R(M),\end{aligned}$\end{center}
where the first one is by Theorem \ref{lem:1.2}, the second one is by \cite[Theorem 2.7]{R}, the third one is since $\bigoplus_{\mathfrak{m}\in\mathrm{Max}R}E(R/\mathfrak{m})$ is injective.

Next assume that $R$ is semi-local.

(2) $\Leftrightarrow$ (3) One has the following equivalences\begin{center}$\begin{aligned}\mathfrak{p}\in\mathrm{Supp}_RD_\mathfrak{m}(M)
&\Longleftrightarrow\mathfrak{p}\in\mathrm{Supp}_RD_\mathfrak{m}(\mathrm{H}_i(M))\ \textrm{for\ some}\ i\\
&\Longleftrightarrow\mathfrak{p}\in\mathrm{coSupp}_R\mathrm{H}_{i}(M)\ \textrm{for\ some}\ i\\
&\Longleftrightarrow\mathfrak{p}\in\mathrm{Supp}_RD_R(\mathrm{H}_{i}(M))\ \textrm{for\ some}\ i\\
&\Longleftrightarrow\mathfrak{p}\in\mathrm{Supp}_RD_R(M),\end{aligned}$\end{center}
where the second equivalence is by \cite[Lemma 2.5]{Y} and Lemma \ref{lem:0.0}, the third one is by \cite[Theorem 2.7]{R}.

(1) $\Leftrightarrow$ (4) One has the following equivalences\begin{center}$\begin{aligned}\mathrm{RHom}_R(R_\mathfrak{p},M^\thicksim)\not\simeq0
&\Longleftrightarrow\prod_{\mathfrak{m}\in\mathrm{Max}R}\mathrm{H}_i(D_\mathfrak{m}(D_\mathfrak{m}(M)_\mathfrak{p}))
\neq0\ \textrm{for\ some}\ i\\
&\Longleftrightarrow\prod_{\mathfrak{m}\in\mathrm{Max}R}D_\mathfrak{m}(D_\mathfrak{m}(\mathrm{H}_{i}(M))_\mathfrak{p})
\neq0\ \textrm{for\ some}\ i\\
&\Longleftrightarrow\prod_{\mathfrak{m}\in\mathrm{Max}R}\mathrm{Hom}_R(R_\mathfrak{p},
D_\mathfrak{m}(D_\mathfrak{m}(\mathrm{H}_{i}(M))))
\neq0\ \textrm{for\ some}\ i\\
&\Longleftrightarrow\mathrm{Hom}_R(R_\mathfrak{p},\mathrm{H}_{i}(M)^\thicksim)\neq0\ \textrm{for\ some}\ i,\end{aligned}$\end{center}
where the second and the third equivalences are because $E(R/\mathfrak{m})$ is injective and $R_\mathfrak{p}$ is flat. Hence Theorem \ref{lem:1.2} and \cite[Theorem 2.15]{Y} imply the desired equivalence.
\end{proof}

Let $\mathcal{U}$ be a subset of $\mathrm{Spec}R$. The specialization closure of $\mathcal{U}$ is the set
\begin{center}$\mathrm{cl}\mathcal{U}=\{\mathfrak{p}\in\textrm{Spec}R\hspace{0.03cm}|\hspace{0.03cm}\textrm{there\ is}\ \mathfrak{q}\in\mathcal{U}\ \textrm{with}\ \mathfrak{q}\subseteq\mathfrak{p}\}$.\end{center}The subset $\mathcal{U}$ is
 specialization closed if $\mathrm{cl}\mathcal{U}=\mathcal{U}$.

\begin{rem}\label{lem:4.3}{\rm $\mathrm{(1)}$ For any $R$-complex $M$, one has that $\mathrm{coSupp}_{R}M=\mathrm{coSupp}_{R}\Sigma M$.

$\mathrm{(2)}$ For an exact triangle $L\rightarrow M\rightarrow N\rightsquigarrow$ in $\mathrm{D}(R)$, we have
\begin{center}$\mathrm{coSupp}_{R}M\subseteq\mathrm{coSupp}_{R}L\cup\mathrm{coSupp}_{R}N$.\end{center}

$\mathrm{(3)}$ For any $R$-complex $M$, the set $\mathrm{coSupp}_RM$ is specialization closed.

(4) $\mathrm{H}({^\mathfrak{p}}M)\cong{^\mathfrak{p}}\mathrm{H}(M)$ for any $\mathfrak{p}\in\mathrm{Spec}R$.

(5) Let $M\in\mathrm{D}^\mathrm{n}_\mathrm{b}(R)$ and $N\in\mathrm{D}(R)$. One has two isomorphisms \begin{center}${^\mathfrak{p}}(M\otimes^\mathrm{L}_RN)\simeq M_\mathfrak{p}\otimes^\mathrm{L}_{R_\mathfrak{p}}{^\mathfrak{p}}N$ and ${^\mathfrak{p}}\mathrm{RHom}_R(M,N)\simeq\mathrm{RHom}_{R_\mathfrak{p}}(M_\mathfrak{p},{^\mathfrak{p}}N)$,\end{center} in $\mathrm{D}(R)$, which implies that \begin{center}$\mathrm{coSupp}_R(M\otimes^\mathrm{L}_RN)\subseteq\mathrm{Supp}_RM\cap\mathrm{coSupp}_RN$,\end{center} \begin{center}$\mathrm{coSupp}_R\mathrm{RHom}_R(M,N)\subseteq\mathrm{Supp}_RM\cap\mathrm{coSupp}_RN$.\end{center}

(6) By Lemma \ref{lem:0.0}, one has that $\mathrm{RHom}_R(R_\mathfrak{p},M^\thicksim)\not\simeq0$ implies that ${^\mathfrak{p}}M\not\simeq0$.

(7) The notion of big cosupport for an $R$-complex $M$ is not the same as the one in \cite{WW}. For example, let $M=R=k[x]$ for any field $k$. Then $\mathrm{Co\textrm{-}supp}_RM=\mathrm{Spec}R$. But $\mathrm{coSupp}_RM=\mathrm{Max}R\neq\mathrm{Spec}R$ by Theorem \ref{lem:1.3}.

(8) Let $M$ be an $R$-complex and $\mathfrak{p}\in\mathrm{Spec}R$. If each $\mathrm{H}_i(M)$ is a Matlis reflexive $R$-module (i.e. $\mathrm{H}_i(M)\cong D_R(D_R(\mathrm{H}_i(M)))$), then $M\simeq D_R(D_R(M))$, and so ${^\mathfrak{p}}D_R(M)\simeq\mathrm{Hom}_{R_\mathfrak{p}}(D_R(D_R(M))_\mathfrak{p},E(k(\mathfrak{p})))\simeq D_{R_\mathfrak{p}}({M_\mathfrak{p}})$. Consequently, \begin{center}$\mathfrak{p}\in\mathrm{Supp}_RM\Longleftrightarrow \mathfrak{p}\in\mathrm{coSupp}_RD_R(M)$.\end{center}}
\end{rem}

In general, we have the following result.

\begin{prop}\label{lem:1.7}{\it{Let $M$ be an $R$-complex.

$\mathrm{(1)}$ If $\mathfrak{p}\in\mathrm{Supp}_RM$, then $\mathfrak{p}\in\mathrm{coSupp}_RD_R(M)$.

$\mathrm{(2)}$ If $M\in\mathrm{D}^\mathrm{n}(R)$, then $\mathfrak{p}\in\mathrm{Supp}_RM$ if and only if $\mathfrak{p}\in\mathrm{coSupp}_RD_R(M)$.}}
\end{prop}
\begin{proof} (1) Since $\mathfrak{p}\in\mathrm{Supp}_RM$, $\mathfrak{p}\in\mathrm{Supp}_R\mathrm{H}_i(M)$ for some $i$, and so $\mathfrak{p}\in\mathrm{coSupp}_R\mathrm{H}_{-i}(D_R(M))$ by Lemma \ref{lem:0.1}. Therefore, $\mathfrak{p}\in\mathrm{coSupp}_RD_R(M)$ by Theorem \ref{lem:1.2}.

(2) ``Only if'' part by (1). ``If'' part. Since $\mathfrak{p}\in\mathrm{coSupp}_RD_R(M)$, $\mathfrak{p}\in\mathrm{coSupp}_R\mathrm{H}_i(D_R(M))$ for some $i$ by Theorem \ref{lem:1.2}, i.e., $\mathfrak{p}\in\mathrm{coSupp}_RD_R(\mathrm{H}_{-i}(M))$. Hence Lemma \ref{lem:0.1}
implies that $\mathfrak{p}\in\mathrm{Supp}_R\mathrm{H}_{-i}(M)$. Consequently, $\mathfrak{p}\in\mathrm{Supp}_RM$.
\end{proof}

The following example shows that the reverse of (1) in the above proposition does not hold in general.

\begin{exa}\label{lem:4.4}{\rm (\cite{Y}) Let $(R,\mathfrak{m},k)$ be a local domain with $\mathrm{dim}R>0$. Consider the complex $M=0\rightarrow\bigoplus_{n>0}R/\mathfrak{m}^n\rightarrow0$. Then $(0)\in\mathrm{Supp}_RD_R(D_R(M))$ and so $(0)\in\mathrm{coSupp}_RD_R(M)$ by Theorem \ref{lem:1.3}. However, $(0)\not\in\mathrm{Supp}_RM$.}
\end{exa}

\bigskip
\section{\bf Another version of small cosupport}
This section introduces the set $\mathrm{cosupp}_RM$ of ``small'' cosupport of an $R$-complex $M$, and provide a duality between the ``small'' cosupport and support as Section 2.

\begin{df}\label{lem:2.1}{\rm Let $M$ be an $R$-complex.
 The ``small'' cosupport of $M$ is defined as
\begin{center}$\mathrm{cosupp}_RM:=\{\mathfrak{p}\in\mathrm{Spec}R\hspace{0.03cm}|\hspace{0.03cm}
\mathrm{RHom}_R(R/\mathfrak{p},{^\mathfrak{p}}M)\not\simeq0\}$.\end{center}}
\end{df}

Next we bring an analogue of Theorem \ref{lem:1.3}.

\begin{thm}\label{lem:2.3}{\it{Let $M$ be an $R$-complex. The following are equivalent:

$\mathrm{(1)}$ $\mathfrak{p}\in\mathrm{cosupp}_RM$;

$\mathrm{(2)}$ $\mathrm{RHom}_R(D_R(M),k(\mathfrak{p}))\not\simeq0$;

$\mathrm{(3)}$ $\mathfrak{p}\in\mathrm{supp}_RD_R(M)$;

$\mathrm{(4)}$ $k(\mathfrak{p})\otimes^\mathrm{L}_{R_\mathfrak{p}}{^\mathfrak{p}}M\not\simeq0$;

$\mathrm{(5)}$ $\mathfrak{p}R_\mathfrak{p}\in\mathrm{cosupp}_{R_\mathfrak{p}}{^\mathfrak{p}}M$.\\
If in addition $R$ is semi-local, then $(1)$--$(5)$ are equivalent to

$\mathrm{(6)}$ $\mathrm{RHom}_R(k(\mathfrak{p}),M^\thicksim)\not\simeq0$;

$\mathrm{(7)}$ $\mathrm{Hom}_R(\coprod_{\mathfrak{m}\in\mathrm{Max}R}D_\mathfrak{m}(M),k(\mathfrak{p}))\not\simeq0$;

$\mathrm{(8)}$ $\mathfrak{p}\in\mathrm{supp}_RD_\mathfrak{m}(M)$ for some $\mathfrak{m}\in\mathrm{Max}R\cap\mathrm{V}(\mathfrak{p})$;

$\mathrm{(9)}$ $k(\mathfrak{p})\otimes^\mathrm{L}_{R_\mathfrak{p}}\mathrm{RHom}_R(R_\mathfrak{p},M^\thicksim)\not\simeq0$.}}
\end{thm}
\begin{proof} One has the following isomorphisms in $\mathrm{D}(R)$:
\begin{center}${^\mathfrak{p}}\mathrm{RHom}_R(R/\mathfrak{p},M)\simeq\mathrm{RHom}_R(D_R(M),k(\mathfrak{p}))
\simeq\mathrm{RHom}_R(R/\mathfrak{p},{^\mathfrak{p}}M)$,\end{center}
\begin{center}$D_R(\mathrm{RHom}_R(R/\mathfrak{p},M))_\mathfrak{p}
\simeq (R/\mathfrak{p}\otimes^\mathrm{L}_R
D_R(M))_\mathfrak{p}\simeq k(\mathfrak{p})\otimes^\mathrm{L}_RD_R(M)$.\end{center}
Hence Theorem \ref{lem:1.3} implies the equivalences of (1)--(3).

(1) $\Leftrightarrow$ (4) This follows from \cite[Fact 3.5]{WW} and the isomorphism $\mathrm{RHom}_R(R/\mathfrak{p},{^\mathfrak{p}}M)\cong
\mathrm{RHom}_{R_\mathfrak{p}}(k(\mathfrak{p}),{^\mathfrak{p}}M)$ in $\mathrm{D}(R)$.

(1) $\Leftrightarrow$ (5) Since $\mathrm{RHom}_{R_\mathfrak{p}}(R_\mathfrak{p}/\mathfrak{p}R_\mathfrak{p},
{^{\mathfrak{p}R_\mathfrak{p}}}({^\mathfrak{p}}M))\simeq D_{R_\mathfrak{p}}(D_{R_\mathfrak{p}}(\mathrm{RHom}_R(R/\mathfrak{p},{^\mathfrak{p}}M)))$, it follows that $\mathrm{RHom}_{R_\mathfrak{p}}(R_\mathfrak{p}/\mathfrak{p}R_\mathfrak{p},
{^{\mathfrak{p}R_\mathfrak{p}}}({^\mathfrak{p}}M))\not\simeq0$ if and only if $\mathrm{RHom}_{R}(R/\mathfrak{p},{^\mathfrak{p}}M)\not\simeq0$, as desired.

One has the following isomorphisms in $\mathrm{D}(R)$:
\begin{center}$\mathrm{RHom}_R(k(\mathfrak{p}),M^\thicksim)\simeq\mathrm{RHom}_R(R_\mathfrak{p},\mathrm{RHom}_R(R/\mathfrak{p},M)^\thicksim)$,\end{center} \begin{center}$\mathrm{Hom}_R(\coprod_{\mathfrak{m}\in\mathrm{Max}R}D_\mathfrak{m}(M),k(\mathfrak{p}))
\simeq\mathrm{Hom}_R(\coprod_{\mathfrak{m}\in\mathrm{Max}R}
D_\mathfrak{m}(\mathrm{RHom}_R(R/\mathfrak{p},M)),E(R/\mathfrak{p}))$,\end{center}
\begin{center}$D_\mathfrak{m}(\mathrm{RHom}_R(R/\mathfrak{p},M))_\mathfrak{p}
\simeq (R/\mathfrak{p}\otimes^\mathrm{L}_R
D_\mathfrak{m}(M))_\mathfrak{p}\simeq k(\mathfrak{p})\otimes^\mathrm{L}_RD_\mathfrak{m}(M)$,\end{center}
Hence Theorem \ref{lem:1.3} implies the equivalences of (1) $\Leftrightarrow$ (6) $\Leftrightarrow$ (7) $\Leftrightarrow$ (8).

(6) $\Leftrightarrow$ (9) This follows from \cite[Fact 3.5]{WW} and the isomorphism $\mathrm{RHom}_R(k(\mathfrak{p}),M^\thicksim)\simeq\mathrm{RHom}_{R_\mathfrak{p}}(k(\mathfrak{p}),
\mathrm{RHom}_R(R_\mathfrak{p},M^\thicksim))$ in $\mathrm{D}(R)$.
\end{proof}

\begin{cor}\label{lem:2.66}{\it{Let $M$ be an $R$-complex. Then $M\not\simeq0$ if and only if $\mathrm{cosupp}_RM\neq\emptyset$.}}
\end{cor}
\begin{proof} $M\not\simeq0$ if and only if $D_R(M)\not\simeq0$ if and only if $\mathrm{supp}_RD_R(M)\neq\emptyset$ if and only if $\mathrm{cosupp}_RM\neq\emptyset$ by Theorem \ref{lem:2.3}.
\end{proof}

\begin{cor}\label{lem:2.99}{\it{Let $M$ be an $R$-complex. One has that \begin{center}$\mathrm{cosupp}_RM=\mathrm{min}(\mathrm{cosupp}_R\mathrm{H}(M))$.\end{center}}}
\end{cor}
\begin{proof} One has the following equivalences \begin{center}$\begin{aligned}\mathfrak{p}\in\mathrm{cosupp}_RM
&\Longleftrightarrow\mathfrak{p}\in\mathrm{supp}_RD_R(M)\\
&\Longleftrightarrow\mathfrak{p}\in\mathrm{min}(\mathrm{supp}_R\mathrm{H}(D_R(M)))\\
&\Longleftrightarrow\mathfrak{p}\in\mathrm{min}(\mathrm{supp}_RD_R(\mathrm{H}(M)))\\
&\Longleftrightarrow\mathfrak{p}\in\mathrm{min}(\mathrm{cosupp}_R\mathrm{H}(M)),\end{aligned}$\end{center}
 where the first and the last equivalences are by Theorem \ref{lem:2.3}, the second one is by \cite[Theorem 5.2]{BIK} and the third one is as $\bigoplus_{\mathfrak{m}\in\mathrm{Max}R}E(R/\mathfrak{m})$ is injective.
\end{proof}

\begin{rem}\label{lem:2.6}{\rm (1) For any $R$-complex $M$, one has ${^\mathfrak{p}}\mathrm{RHom}_R(R/\mathfrak{p},M)
\simeq\mathrm{RHom}_R(R/\mathfrak{p},{^\mathfrak{p}}M)$. Hence $\mathfrak{p}\in\mathrm{cosupp}_RM\Longleftrightarrow\mathfrak{p}\in\mathrm{coSupp}_R\mathrm{RHom}_R(R/\mathfrak{p},M)$.

(2) If $M$ is an $R$-module, then $\mathrm{cosupp}_RM=\{\mathfrak{p}\in\mathrm{Spec}R\hspace{0.03cm}|\hspace{0.03cm}
{^\mathfrak{p}}\mathrm{Ext}^i_{R}(R/\mathfrak{p},M)\neq0\ \textrm{for\ some}\ i\}$.

(3) Let $\mathrm{V}$ be a specialization closed subset of $\mathrm{Spec}R$. For each $R$-module $M$, one has \begin{center}$\mathrm{cosupp}_RM\subseteq\mathrm{V}\Longleftrightarrow{^\mathfrak{p}}M=0$ for each $\mathfrak{p}\in\mathrm{Spec}R\backslash\mathrm{V}$.\end{center}

(4) For each $R$-module $M$, one has inclusions \begin{center}$\mathrm{cosupp}_RM\subseteq\mathrm{cl}(\mathrm{cosupp}_RM)=\mathrm{coSupp}_RM
\subseteq\mathrm{V}(\mathrm{Ann}_RM)$.\end{center}}
\end{rem}

\begin{prop}\label{lem:2.7}{\it{$\mathrm{(1)}$ Let $M\in\mathrm{D}^\mathrm{f}_\mathrm{b}(R)$ and $N\in\mathrm{D}(R)$. One has that \begin{center}$\mathrm{cosupp}_R\mathrm{RHom}_R(M,N)=\mathrm{supp}_RM\cap\mathrm{cosupp}_RN$.\end{center}

$\mathrm{(2)}$ Let $M\in\mathrm{D}^\mathrm{f}_+(R)$ and $N\in\mathrm{D}_+(R)$ or $M\in\mathrm{D}^\mathrm{f}_\mathrm{b}(R)$ and $N\in\mathrm{D}(R)$. One has that \begin{center}$\mathrm{cosupp}_R(M\otimes^\mathrm{L}_RN)=\mathrm{supp}_RM\cap\mathrm{cosupp}_RN$.\end{center}
In particular, for any ideal $\mathfrak{a}$ of $R$ and an arbitrary $R$-complex $M$, we have \begin{center}$\mathrm{cosupp}_R\mathrm{RHom}_R(R/\mathfrak{a},M)=
\mathrm{cosupp}_RM\cap\mathrm{V}(\mathfrak{a})=\mathrm{cosupp}_R(R/\mathfrak{a}\otimes^\mathrm{L}_RM)$.\end{center}}}
\end{prop}
\begin{proof} (1) One has the following equivalences\begin{center}$\begin{aligned}\mathfrak{p}\in\mathrm{cosupp}_R\mathrm{RHom}_R(M,N)
&\Longleftrightarrow\mathfrak{p}\in\mathrm{supp}_RD_R(\mathrm{RHom}_R(M,N))\\
&\Longleftrightarrow\mathfrak{p}\in\mathrm{supp}_R(M\otimes^\mathrm{L}_RD_R(N))\\
&\Longleftrightarrow\mathfrak{p}\in\mathrm{supp}_RM\cap\mathrm{supp}_RD_R(N)\\
&\Longleftrightarrow\mathfrak{p}\in\mathrm{supp}_RM\cap\mathrm{cosupp}_RN,\end{aligned}$\end{center}where the first and the fourth equivalences are by Theorem \ref{lem:2.3}, the second one is by \cite[Theorem 2.5.6]{CF} and the third one is by \cite[Proposition 3.12]{WW}.

(2) One has the following equivalences\begin{center}$\begin{aligned}\mathfrak{p}\in\mathrm{cosupp}_R(M\otimes^\mathrm{L}_RN)
&\Longleftrightarrow\mathfrak{p}\in\mathrm{supp}_RD_R(M\otimes^\mathrm{L}_RN)\\
&\Longleftrightarrow\mathfrak{p}\in\mathrm{supp}_R\mathrm{RHom}_R(M,D_R(N))\\
&\Longleftrightarrow\mathfrak{p}\in\mathrm{supp}_RM\cap\mathrm{supp}_RD_R(N)\\
&\Longleftrightarrow\mathfrak{p}\in\mathrm{supp}_RM\cap\mathrm{cosupp}_RN,\end{aligned}$\end{center}where the first and the fourth equivalences are by Theorem \ref{lem:2.3}, the third one is by \cite[Proposition 3.16]{WW}.
\end{proof}

The following proposition is an analogue of Proposition \ref{lem:1.7}.

\begin{prop}\label{lem:2.2}{\it{Let $M$ be an $R$-complex.

$\mathrm{(1)}$ If $\mathfrak{p}\in\mathrm{supp}_RM$, then $\mathfrak{p}\in\mathrm{cosupp}_RD_R(M)$.

$\mathrm{(2)}$ If $M\in\mathrm{D}^\mathrm{n}(R)$, then $\mathfrak{p}\in\mathrm{supp}_RM$ if and only if $\mathfrak{p}\in\mathrm{cosupp}_RD_R(M)$.}}
\end{prop}
\begin{proof} (1)
Let $\mathfrak{p}\in\mathrm{supp}_RM$. Then $\mathfrak{p}\in\mathrm{Supp}_R(R/\mathfrak{p}\otimes^\mathrm{L}_RM)$, and so $\mathfrak{p}\in\mathrm{coSupp}_RD_R(R/\mathfrak{p}\otimes^\mathrm{L}_RM)$ by Proposition \ref{lem:1.7}(1). But
$D_R(R/\mathfrak{p}\otimes^\mathrm{L}_RM)\simeq\mathrm{RHom}_R(R/\mathfrak{p},D_R(M))$, it follows from Remark \ref{lem:2.6}(1) that $\mathfrak{p}\in\mathrm{cosupp}_RD_R(M)$.

(2) This follows from Proposition \ref{lem:1.7}(2) since $R/\mathfrak{p}\otimes^\mathrm{L}_RM\in\mathrm{D}^\mathrm{n}(R)$.
\end{proof}

\bigskip
\section{\bf Relations between big and small cosupport}
We devote this section to some relations between ``big'' and ``small'' cosupport. We show that $\mathrm{cosupp}_RM\subseteq\mathrm{coSupp}_RM$ and the inclusion may be strict.

\begin{prop}\label{lem:2.12}{\it{Let $M$ be an $R$-complex. The sets $\mathrm{supp}_RM$ and $\mathrm{cosupp}_RM$ have the same maximal elements with respect to containment, i.e.,
$\mathrm{max}(\mathrm{supp}_RM)=\mathrm{max}(\mathrm{cosupp}_RM)$. Moreover, $\mathrm{max}(\mathrm{cosupp}_RM)=\mathrm{max}(\mathrm{co\textrm{-}supp}_RM)$.}}
\end{prop}
\begin{proof} We prove that $\mathrm{max}(\mathrm{supp}_RM)\subseteq\mathrm{cosupp}_RM$ and $\mathrm{max}(\mathrm{cosupp}_RM)\subseteq\mathrm{supp}_RM$.

If $\mathfrak{p}\in\mathrm{max}(\mathrm{supp}_RM)$, then $\mathrm{co\textrm{-}supp}_R(R/\mathfrak{p}\otimes^\mathrm{L}_RD_R(M))=\{\mathfrak{p}\}$ by \cite[Proposition 4.10]{WW}. Hence $\mathrm{RHom}_R(D_R(M),k(\mathfrak{p}))\simeq\mathrm{RHom}_R(R/\mathfrak{p}\otimes^\mathrm{L}_RD_R(M),
E(R/\mathfrak{p}))\not\simeq0$ by \cite[Proposition 5.4]{BIK2} and so $\mathfrak{p}\in\mathrm{cosupp}_RM$ by Theorem \ref{lem:2.3}.
 If $\mathfrak{p}\in\mathrm{max}(\mathrm{cosupp}_RM)$, then $\mathrm{cosupp}_R(R/\mathfrak{p}\otimes^\mathrm{L}_RM)=\{\mathfrak{p}\}$, so $\mathfrak{p}\in\mathrm{max}(\mathrm{supp}_RD_R(R/\mathfrak{p}\otimes^\mathrm{L}_RM))$. Thus \cite[Proposition 4.7(b)]{WW} implies that
$\mathfrak{p}\in\mathrm{co\textrm{-}supp}_R\mathrm{RHom}_R(R/\mathfrak{p},D_R(M))$. Consequently,
$\mathfrak{p}\in\mathrm{supp}_RM$ by \cite[Proposition 4.10]{WW}.

The second statement follows from \cite[Theorem 4.13]{BIK2}.
\end{proof}

\begin{prop}\label{lem:2.4}{\it{For every $R$-complex $M$, one has an inclusion $\mathrm{cosupp}_RM\subseteq\mathrm{coSupp}_RM$; equality holds if $R$ is a semi-local complete ring and
 $M\in\mathrm{D}^\mathrm{a}_-(R)$.}}
\end{prop}
\begin{proof} The inclusion follows from Theorems \ref{lem:1.3} and \ref{lem:2.3} since $\mathrm{supp}_RD_R(M)\subseteq\mathrm{Supp}_RD_R(M)$. Now let $M\in\mathrm{D}^\mathrm{a}_-(R)$ and $\mathfrak{p}\in\mathrm{coSupp}_RM$, $i=\mathrm{inf}{^\mathfrak{p}}M$. Then ${^\mathfrak{p}}M\in\mathrm{D}^\mathrm{a}_-(R_\mathfrak{p})$ by \cite[Theorem 2.3]{R}, and so $\mathrm{H}_i(\mathrm{RHom}_{R_\mathfrak{p}}(k(\mathfrak{p}),{^\mathfrak{p}}M))\cong
\mathrm{Hom}_{R_\mathfrak{p}}(k(\mathfrak{p}),\mathrm{H}_i({^\mathfrak{p}}M))\neq0$ by \cite[Theorem 4.3]{Y}.
Consequently, $\mathrm{RHom}_{R}(R/\mathfrak{p},{^\mathfrak{p}}M)\not\simeq0$ and $\mathfrak{p}\in\mathrm{cosupp}_RM$, as claimed.
\end{proof}

The next example shows that the inclusion in the proposition \ref{lem:2.4} may be
strict:

\begin{exa}\label{lem:2.13}{\rm (\cite[Example 9.4]{BIK}) Let $k$ be a field and $R=k[[x, y]]$ the power series ring in indeterminates
$x, y$, and set $\mathfrak{m}=(x, y)$ the maximal ideal of $R$. The minimal injective resolution of $R$ has
the form: \begin{center}$\cdots\rightarrow0\rightarrow Q\rightarrow\coprod_{\mathrm{ht}\mathfrak{p}=1}E(R/\mathfrak{p})\rightarrow E(R/\mathfrak{m})\rightarrow0\rightarrow\cdots$,\end{center}
where $Q$ denotes the fraction field of $R$. Let $M$ denote the truncated complex
 \begin{center}$\cdots\rightarrow0\rightarrow Q\rightarrow\coprod_{\mathrm{ht}\mathfrak{p}=1}E(R/\mathfrak{p})\rightarrow0\rightarrow\cdots$.\end{center}One has that $\mathrm{coSupp}_RD_R(M)=\mathrm{Spec}R$ since $\mathrm{Spec}R=\mathrm{Supp}_RM\subseteq\mathrm{coSupp}_RD_R(M)$. But $\mathfrak{m}\not\in\mathrm{cosupp}_RD_R(M)$. In fact, if $\mathfrak{m}\in\mathrm{cosupp}_RD_R(M)$ then $\mathfrak{m}\in\mathrm{supp}_RD_R(M)$ by Proposition \ref{lem:2.12}, and hence $\mathfrak{m}\in\mathrm{cosupp}_RM$ by Theorem \ref{lem:2.3}. Consequently, $\mathfrak{m}\in\mathrm{supp}_RM$ by Proposition \ref{lem:2.12} again, which is a contradiction since $\mathrm{supp}_RM=\mathrm{Spec}R\backslash\{\mathfrak{m}\}$.}
\end{exa}

\begin{cor}\label{lem:2.9}{\it{Let $M$ be an $R$-complex. The sets $\mathrm{cosupp}_RM$ and $\mathrm{coSupp}_RM$ have the same minimal elements with respect to containment, i.e. $\mathrm{min}(\mathrm{cosupp}_RM)=\mathrm{min}(\mathrm{coSupp}_RM)$.}}
\end{cor}
\begin{proof} This follows from Theorems \ref{lem:1.3}, \ref{lem:2.3} and \cite[Proposition 3.14]{WW}.
\end{proof}

\begin{cor}\label{lem:2.10}{\it{Let $M$ be an $R$-complex.

$(1)$ For an ideal $\mathfrak{a}$ of $R$, $\mathrm{coSupp}_RM\subseteq\mathrm{V}(\mathfrak{a})$ if and only if $\mathrm{cosupp}_RM\subseteq\mathrm{V}(\mathfrak{a})$.

$(2)$ The Zariski closures of $\mathrm{coSupp}_RM$ and $\mathrm{cosupp}_RM$ are equal.}}
\end{cor}
\begin{proof} This follows from Theorems \ref{lem:1.3}, \ref{lem:2.3} and \cite[Proposition 3.15]{WW}.
\end{proof}

\begin{prop}\label{lem:4.9}{\it{$\mathrm{(1)}$ If $M$ is in $\mathrm{D}^\mathrm{n}(R)$, then $\mathrm{cosupp}_RM\subseteq\mathrm{co\textrm{-}supp}_RM$ and $\mathrm{coSupp}_RM\subseteq\mathrm{Co\textrm{-}supp}_RM$.

$\mathrm{(2)}$ Assume that $R$ is a semi-local ring and $M\in\mathrm{D}(R)$. If each $\mathrm{H}_i(M)$ is a Matlis reflexive $R$-module, then $\mathrm{co\textrm{-}supp}_RM=\mathrm{cosupp}_RM$ and $\mathrm{Co\textrm{-}supp}_RM=\mathrm{coSupp}_RM$.}}
\end{prop}
\begin{proof} (1) Since $M\in\mathrm{D}^\mathrm{n}(R)$, it follows that $\mathrm{cosupp}_RM\subseteq\mathrm{coSupp}_RM\subseteq\mathrm{Max}R$. Hence Proposition \ref{lem:2.12} implies that $\mathrm{cosupp}_RM\subseteq\mathrm{co\textrm{-}supp}_RM$. Note that $\mathrm{cosupp}_RM=\mathrm{coSupp}_RM$ and $\mathrm{co\textrm{-}supp}_RM\subseteq\mathrm{Co\textrm{-}supp}_RM$, so $\mathrm{coSupp}_RM\subseteq\mathrm{Co\textrm{-}supp}_RM$.

(2) Since $\mathrm{H}_i(M)\cong D_R(D_R(\mathrm{H}_i(M)))$ for all $i$, it follows that $M\simeq D_R(D_R(M))$. Hence
 $\mathrm{co\textrm{-}supp}_RM=\mathrm{co\textrm{-}supp}_RD_R(D_R(M))=\mathrm{supp}_RD_R(M)=
 \mathrm{cosupp}_RM$ and $\mathrm{Co\textrm{-}supp}_RM=\mathrm{Co\textrm{-}supp}_RD_R(D_R(M))=\mathrm{Supp}_RD_R(M)
 =\mathrm{coSupp}_RM$.
\end{proof}

\begin{cor}\label{lem:4.10}{\it{Assume that $R$ is a semi-local complete ring. If $M\in\mathrm{D}^\mathrm{n}(R)$ or $M\in\mathrm{D}^\mathrm{a}(R)$, then $\mathrm{co\textrm{-}supp}_RM=\mathrm{cosupp}_RM$ and $\mathrm{Co\textrm{-}supp}_RM=\mathrm{coSupp}_RM$.}}
\end{cor}

The example in Remark \ref{lem:4.3}(7) shows that the inclusion in Proposition \ref{lem:4.9} may be strict.

\bigskip
\section{\bf Computations of cosupport and support}
This section puts emphasis on computing the ``small'' support and ``small'' cosupport, and studying the relation between $\mathrm{cosupp}_RM$ and $\mathrm{cosupp}_R\mathrm{H}(M)$. As an application, we give the comparison of the ``small'' support and cosupport.

\begin{prop}\label{lem:3.3}{\it{Let $\mathfrak{p}$ be a point in $\mathrm{Spec}R$. One has that

$\mathrm{(1)}$ $\mathrm{cosupp}_RR=\mathrm{Max}R$ and $\mathrm{supp}_RR=\mathrm{Spec}R$.

$\mathrm{(2)}$ $\mathrm{cosupp}_Rk(\mathfrak{p})=\{\mathfrak{p}\}=\mathrm{supp}_Rk(\mathfrak{p})$.

$\mathrm{(3)}$ $\mathrm{supp}_RE(R/\mathfrak{p})=\{\mathfrak{p}\}$ and $\mathrm{cosupp}_RE(R/\mathfrak{p})=\mathrm{U}(\mathfrak{p})$.}}
\end{prop}
\begin{proof} (1) It follows from Theorem \ref{lem:2.3} and \cite[Proposition 3.11]{WW} that \begin{center}$\mathrm{cosupp}_RR=\mathrm{supp}_RD_R(R)=
\mathrm{supp}_R\bigoplus_{\mathfrak{m}\in\mathrm{Max}R}E(R/\mathfrak{m})=\mathrm{Max}R$.\end{center}
It follows from Proposition \ref{lem:2.2} that \begin{center}$\mathrm{supp}_RR=\mathrm{cosupp}_RD_R(R)=
\mathrm{cosupp}_R\bigoplus_{\mathfrak{m}\in\mathrm{Max}R}E(R/\mathfrak{m})=\mathrm{Spec}R$.\end{center}

(1) Since $\mathrm{supp}_Rk(\mathfrak{p})=\{\mathfrak{p}\}$, it follows from Proposition \ref{lem:2.12} that $\mathrm{cosupp}_Rk(\mathfrak{p})\subseteq\mathrm{U}(\mathfrak{p})$. On the other hand, $\mathrm{cosupp}_Rk(\mathfrak{p})=\mathrm{cosupp}_R(R/\mathfrak{p}\otimes_RR_\mathfrak{p})
\subseteq\mathrm{V}(\mathfrak{p})$ by Proposition \ref{lem:2.7}. Consequently, $\mathrm{cosupp}_Rk(\mathfrak{p})=\{\mathfrak{p}\}$.

(2) This follows from \cite[Corollary 2.18]{Y} and \cite[Proposition 6.3]{WW}.
\end{proof}

\begin{rem}\label{lem:3.15}{\rm (i) Example \ref{lem:2.13} shows that $\mathrm{supp}_RM$ and $\mathrm{supp}_R\mathrm{H}(M)$ need not coincide and $\mathrm{cosupp}_RM$ and $\mathrm{cosupp}_R\mathrm{H}(M)$ need not coincide.

(ii) For any $R$-complex $M$, $\mathrm{cosupp}_RM$ may be not a specialization closed subset.}
\end{rem}

The next results study relations between $\mathrm{cosupp}_RM$ (resp. $\mathrm{supp}_RM$)and $\mathrm{cosupp}_R\mathrm{H}(M)$ (resp. $\mathrm{supp}_R\mathrm{H}(M)$).

\begin{prop}\label{lem:3.9}{\it{$\mathrm{(1)}$ For each $M\in\mathrm{D}^\mathrm{n}_+(R)$, one has $\mathrm{supp}_RM=\bigcup_{i\in\mathbb{Z}}\mathrm{supp}_R\mathrm{H}_i(M)$.

$\mathrm{(2)}$ If $R$ is semi-local complete, then for $M\in\mathrm{D}^\mathrm{a}_-(R)$, $\mathrm{cosupp}_RM=\bigcup_{i\in\mathbb{Z}}\mathrm{cosupp}_R\mathrm{H}_i(M)$.}}
\end{prop}
\begin{proof} We just prove one of the statements since the other is dual.

By Proposition \ref{lem:2.4}, $\mathrm{cosupp}_RM=\mathrm{coSupp}_RM$. But $\mathrm{coSupp}_RM=\bigcup_{i\in\mathbb{Z}}\mathrm{coSupp}_R\mathrm{H}_i(M)=
\bigcup_{i\in\mathbb{Z}}\mathrm{cosupp}_R\mathrm{H}_i(M)$ by Theorem \ref{lem:1.2}, as desired.
\end{proof}

\begin{prop}\label{lem:3.10}{\it{$\mathrm{(1)}$ For each $M\in\mathrm{D}_-(R)$, one has $\mathrm{supp}_RM\subseteq\bigcup_{i\in\mathbb{Z}}\mathrm{supp}_R\mathrm{H}_i(M)$.

$\mathrm{(2)}$ For each $M\in\mathrm{D}_+(R)$, one has $\mathrm{cosupp}_RM\subseteq\bigcup_{i\in\mathbb{Z}}\mathrm{cosupp}_R\mathrm{H}_i(M)$.}}
\end{prop}
\begin{proof} We just prove (1) since (2) follows by duality,

 First let $M\in\mathrm{D}_\mathrm{b}(R)$. If $\mathrm{inf}M=\mathrm{sup}M=r$, then $M\simeq\Sigma^r\mathrm{H}_r(M)$ and $\mathrm{supp}_RM\subseteq\mathrm{supp}_R\mathrm{H}_r(M)$.
 Assume that $\mathrm{sup}M-\mathrm{inf}M>0$. The exact triangle $\sigma_{\geq\mathrm{inf}M+1}(M)\rightarrow M\rightarrow \Sigma^{\mathrm{inf}M}\mathrm{H}_{\mathrm{inf}M}(M)\rightsquigarrow$ yields that \begin{center}$\mathrm{supp}_RM\subseteq\mathrm{supp}_R\sigma_{\geq\mathrm{inf}M+1}(M)\cup
 \mathrm{supp}_R\mathrm{H}_{\mathrm{inf}M}(M)$.\end{center} But  $\mathrm{supp}_R\sigma_{\geq\mathrm{inf}M+1}(M)\subseteq\bigcup_{i\in\mathbb{Z}}\mathrm{supp}_R
 \mathrm{H}_i(\sigma_{\geq\mathrm{inf}M+1}(M))=\bigcup_{i\geq\mathrm{inf}M+1}\mathrm{supp}_R
 \mathrm{H}_i(M)$ by induction, so $\mathrm{supp}_RM\subseteq\bigcup_{i\in\mathbb{Z}}\mathrm{supp}_R\mathrm{H}_i(M)$. Now let $M\in\mathrm{D}_-(R)$. Then $M=\underrightarrow{\textrm{lim}}\sigma_{\geq n}(M)$. Since $\mathrm{supp}_RM\subseteq\bigcup_{n\leq0}\mathrm{supp}_R\sigma_{\geq n}(M)$ and $\mathrm{supp}_R\sigma_{\geq n}(M)\subseteq\bigcup_{i\geq n}\mathrm{supp}_R\mathrm{H}_i(M)$, it follows that $\mathrm{supp}_RM\subseteq\bigcup_{i\in\mathbb{Z}}\mathrm{supp}_R\mathrm{H}_i(M)$.
\end{proof}

The following corollary is a generalization of \cite[Theorem 6.7]{WW}.

\begin{cor}\label{lem:3.8}{\it{$\mathrm{(1)}$ For each $M\in\mathrm{D}^\mathrm{n}_+(R)$, one has that $\mathrm{cosupp}_RM\subseteq\mathrm{supp}_RM$.

$\mathrm{(2)}$ If $R$ is a semi-local complete ring, then for $M\in\mathrm{D}^\mathrm{a}_-(R)$, $\mathrm{supp}_RM\subseteq\mathrm{cosupp}_RM$.}}
\end{cor}
\begin{proof} We just prove (1) since (2) follows by duality.

By Proposition \ref{lem:3.10} (2), $\mathrm{cosupp}_RM\subseteq\bigcup_{i\in\mathbb{Z}}\mathrm{cosupp}_R\mathrm{H}_i(M)$. But $\mathrm{H}_i(M)$ is noetherian and $\mathrm{coSupp}_R\mathrm{H}_i(M)\subseteq\mathrm{Max}R$, it follows from Propositions \ref{lem:2.12} and \ref{lem:3.9}  that $\bigcup_{i\in\mathbb{Z}}\mathrm{cosupp}_R\mathrm{H}_i(M)\subseteq
\bigcup_{i\in\mathbb{Z}}\mathrm{supp}_R\mathrm{H}_i(M)=\mathrm{supp}_RM$, as claimed.
\end{proof}

\begin{rem}\label{lem:3.14}{\rm (i) The assumption $M\in\mathrm{D}^\mathrm{n}_+(R)$ in (1) and  $M\in\mathrm{D}^\mathrm{a}_-(R)$ in (2) in Corollary \ref{lem:3.8} are essential. For example, assume that $(R,\mathfrak{m})$ is local and not artinian. One has \begin{center}$\mathrm{supp}_RE(R/\mathfrak{m})=\{\mathfrak{m}\}\subsetneq\mathrm{Spec}R=
\mathrm{cosupp}_RE(R/\mathfrak{m})$,\end{center}
\begin{center}$\mathrm{cosupp}_RR=\{\mathfrak{m}\}\subsetneq\mathrm{Spec}R=
\mathrm{supp}_RR$.\end{center}

(ii) Proposition \ref{lem:3.3} (1) and (3) show that one can has proper containment or equality in the above corollary.}
\end{rem}

\bigskip
\section{\bf Coassociated prime for complexes}
The aim of this section is to develop a theory dual that of associated primes of complexes introduced by Christensen in \cite{C}, and find an extension of Nakayama lemma.

Let $(R,\mathfrak{m},k)$ be a local ring and $M$ an $R$-complex. The depth of $M$ is
\begin{center}$\mathrm{depth}_RM=-\mathrm{sup}\mathrm{RHom}_R(k,M)$.\end{center}Following \cite{C}, we say that $\mathfrak{p}\in\mathrm{Spec}R$ is a associated prime ideal for $M\in\mathrm{D}_-(R)$ if $\mathrm{depth}_{R_\mathfrak{p}}M_\mathfrak{p}=-\mathrm{sup}M_\mathfrak{p}<\infty$.
For $M\not\simeq0$ in $\mathrm{D}_-(R)$, we set
\begin{center}$\mathrm{ass}_RM=\mathrm{Ass}_R\mathrm{H}_{\mathrm{sup}M}(M)$ and $\mathrm{z}_RM=\mathrm{z}_R\mathrm{H}_{\mathrm{sup}M}(M)$ and $\mathrm{Z}_RM=\bigcup_{\mathfrak{p}\in\mathrm{Ass}_RM}\mathfrak{p}$.\end{center}

Let $K$ be an $R$-module. A prime ideal $\mathfrak{p}$ of $R$ is called
a coassoczated prime of $K$ if there exists a cocyclic homomorphic image
$L$ of $K$ such that $\mathfrak{p}=\mathrm{Ann}_RL$. The set of coassociated prime ideals of
$K$ is denoted by $\mathrm{Coass}_RK$.

For an $R$-module $K$ the subset $\mathrm{w}_RK$ of $R$ is defined by
\begin{center}$\mathrm{w}_RK=\{r\in R\hspace{0.03cm}|\hspace{0.03cm}K\stackrel{r\cdot}\longrightarrow K\ \textrm{is\ not\ surjective}\}$.\end{center} By \cite[Theorem 1.13]{Y}, $\mathrm{w}_RK=\bigcup_{\mathfrak{p}\in\mathrm{Coass}_RK}\mathfrak{p}$.

Let $(R,\mathfrak{m},k)$ be a local ring and $M$ an $R$-complex.  The width of $M$ is
\begin{center}$\mathrm{width}_RM=\mathrm{inf}(k\otimes^\mathrm{L}_RM)$.\end{center}

\begin{df}\label{lem:5.1}{\rm (1) We say that $\mathfrak{p}\in\mathrm{Spec}R$ is a coassociated prime ideal for $M\in\mathrm{D}_+(R)$ if $\mathrm{width}_{R_\mathfrak{p}}{^\mathfrak{p}}M=\mathrm{inf}{^\mathfrak{p}}M>-\infty$, that is, $\mathrm{Coass}_RM=\{\mathfrak{p}\in\mathrm{coSupp}_RM\hspace{0.03cm}|\hspace{0.03cm}
\mathrm{width}_{R_\mathfrak{p}}{^\mathfrak{p}}M=\mathrm{inf}{^\mathfrak{p}}M\}$.

(2) For an $R$-complex $M\not\simeq0$ in $\mathrm{D}_+(R)$, we set
\begin{center}$\mathrm{coass}_RM=\mathrm{Coass}_R\mathrm{H}_{\mathrm{inf}M}(M)$ and $\mathrm{w}_RM=\mathrm{w}_R\mathrm{H}_{\mathrm{inf}M}(M)$ and $\mathrm{W}_RM=\bigcup_{\mathfrak{p}\in\mathrm{Coass}_RM}\mathfrak{p}$,\end{center} and for $M\simeq0$ we set $\mathrm{coass}_RM=\emptyset$ and $\mathrm{w}_RM=\emptyset$.}
\end{df}

\begin{thm}\label{lem:5.5}{\it{Let $M\in\mathrm{D}_+(R)$. Then
$\mathfrak{p}\in\mathrm{Coass}_RM$ if and only if
$\mathfrak{p}\in\mathrm{Ass}_RD_R(M)$. In particular, $M\not\simeq0$ if and only if $\mathrm{Coass}_RM\neq\emptyset$.}}
\end{thm}
\begin{proof} Since ${^\mathfrak{p}}M=\mathrm{Hom}_{R_\mathfrak{p}}(D_R(M)_\mathfrak{p},
E_{R_\mathfrak{p}}(k(\mathfrak{p})))$, it follows that $-\mathrm{sup}D_R(M)_\mathfrak{p}=\mathrm{inf}{^\mathfrak{p}}M=i$ is finite. One has the following equivalences
\begin{center}$\begin{aligned}\mathfrak{p}\in\mathrm{Coass}_RM
&\Longleftrightarrow\mathrm{inf}(k(\mathfrak{p})\otimes^\mathrm{L}_{R_\mathfrak{p}}{^\mathfrak{p}}M)
=\mathrm{inf}{^\mathfrak{p}}M=i\\
&\Longleftrightarrow k(\mathfrak{p})\otimes_{R_\mathfrak{p}}\mathrm{H}_i({^\mathfrak{p}}M)\neq0\\
&\Longleftrightarrow k(\mathfrak{p})\otimes_{R_\mathfrak{p}}\mathrm{Hom}_{R_\mathfrak{p}}(D_R(\mathrm{H}_i(M))_\mathfrak{p},
E(k(\mathfrak{p}))\neq0\\
&\Longleftrightarrow\mathrm{Hom}_{R_\mathfrak{p}}(\mathrm{Hom}_{R_\mathfrak{p}}(k(\mathfrak{p}),
D_R(\mathrm{H}_i(M))_\mathfrak{p}),E(k(\mathfrak{p}))\neq0\\
&\Longleftrightarrow \mathrm{H}_{-i}(\mathrm{RHom}_{R_\mathfrak{p}}(k(\mathfrak{p}),D_R(M)_\mathfrak{p})=\mathrm{Hom}_{R_\mathfrak{p}}(k(\mathfrak{p}),\mathrm{H}_{-i}(D_R(M)_\mathfrak{p}))\neq0\\
&\Longleftrightarrow\mathfrak{p}R_\mathfrak{p}\in\mathrm{Ass}_{R_\mathfrak{p}}\mathrm{H}_{-i}(D_R(M)_\mathfrak{p})\\
&\Longleftrightarrow\mathfrak{p}\in\mathrm{Ass}_{R}D_R(M),\end{aligned}$\end{center}
where the second one is by \cite[Lemma 2.4.14]{CF}, the third one is since $E(R/\mathfrak{m})$ and $E(k(\mathfrak{p})$ are injective and $R_\mathfrak{p}$ is flat, the fourth one is by \cite[Theorem 2.5.6]{CF}, the fifth one is since $E(k(\mathfrak{p}))$ is faithful injective and the last one is by \cite[Observations 2.4]{C}.
\end{proof}

\begin{rem}\label{lem:5.2}{\rm (1) Let $K$ be an $R$-module. By Theorem \ref{lem:5.5}, $\mathfrak{p}\in\mathrm{Coass}_RK$ if and only if
$\mathfrak{p}\in\mathrm{Ass}_RD_R(K)$ if and only if $\mathfrak{p}R_\mathfrak{p}\in\mathrm{Ass}_{R_\mathfrak{p}}D_R(K)_\mathfrak{p}$ if and only if
$\mathfrak{p}R_\mathfrak{p}\in\mathrm{Coass}_{R_\mathfrak{p}}{^\mathfrak{p}}K$ since the morphism $k(\mathfrak{p})\rightarrow D_R(K)_\mathfrak{p}$ is injective if and only if the morphism ${^\mathfrak{p}}K\rightarrow k(\mathfrak{p})\cong\mathrm{Hom}_{R_\mathfrak{p}}(k(\mathfrak{p}),E_{R_\mathfrak{p}}(k(\mathfrak{p})))$ is surjective.

(2) Let $M\in\mathrm{D}_+(R)$ and $\mathfrak{p}\in\mathrm{coSupp}_RM$ and set $\mathrm{inf}{^\mathfrak{p}}M=i$. Then
\begin{center}$\begin{aligned}\mathfrak{p}\in\mathrm{Coass}_RM
&\Longleftrightarrow \mathfrak{p}\in\mathrm{Ass}_RD_R(M)\\
&\Longleftrightarrow
\mathfrak{p}R_\mathfrak{p}\in\mathrm{Ass}_{R_\mathfrak{p}}D_R(\mathrm{H}_i(M))_\mathfrak{p}\\
&\Longleftrightarrow
\mathfrak{p}R_\mathfrak{p}\in\mathrm{Coass}_{R_\mathfrak{p}}\mathrm{H}_i({^\mathfrak{p}}M)\\
&\Longleftrightarrow
\mathfrak{p}R_\mathfrak{p}\in\mathrm{coass}_{R_\mathfrak{p}}{^\mathfrak{p}}M\\
&\Longleftrightarrow\mathfrak{p}\in\mathrm{Coass}_{R}\mathrm{H}_{i}(M).\end{aligned}$\end{center}
In particular there exists the following inclusion
\begin{center}$\mathrm{coass}_RM\subseteq\mathrm{Coass}_RM$.\end{center}Also $\mathrm{w}_RM=\mathrm{z}_RD_R(M)\subseteq\mathrm{Z}_RD_R(M)=\mathrm{W}_RM$.

(3) Let $M\in\mathrm{D}_+(R)$. Then every minimal prime ideal in $\mathrm{coSupp}_RM$ belongs
to $\mathrm{Coass}_RM$,
and $\mathrm{Coass}_RM\subseteq\mathrm{cosupp}_RM$ for $M\in\mathrm{D}_\mathrm{b}(R)$ by \cite[Proposition 2.6]{C}.

(4) If $M\in\mathrm{D}^\mathrm{a}_\mathrm{b}(R)$ then the set of minimal prime ideals in $\mathrm{coSupp}_RM$ is finite.

(5) If $M$ is an $R$-module, then $\mathrm{coass}_RM=\mathrm{Coass}_RM$ and $\mathrm{w}_RM=\mathrm{W}_RM$.}
\end{rem}

\begin{prop}\label{lem:5.7}{\it{$\mathrm{(1)}$ Let $M\in\mathrm{D}^\mathrm{f}_\mathrm{b}(R)$ and $N\in\mathrm{D}_-(R)$. One has that \begin{center}$\mathrm{Ass}_R\mathrm{RHom}_R(M,N)=\mathrm{Supp}_RM\cap\mathrm{Ass}_RN$.\end{center}

$\mathrm{(2)}$ Let $M\in\mathrm{D}^\mathrm{f}_\mathrm{b}(R)$ and $N\in\mathrm{D}_+(R)$. One has that \begin{center}$\mathrm{Coass}_R(M\otimes^\mathrm{L}_RN)=\mathrm{Supp}_RM\cap\mathrm{Coass}_RN$.\end{center}}}
\end{prop}
\begin{proof} (1) There exists a series of equivalences \begin{center}$\begin{aligned}\mathfrak{p}\in\mathrm{Ass}_R\mathrm{RHom}_R(M,N)
&\Longleftrightarrow-\mathrm{sup}\mathrm{RHom}_R(M,N)_\mathfrak{p}=
\mathrm{depth}_{R_\mathfrak{p}}\mathrm{RHom}_R(M,N)_\mathfrak{p}\\
&\Longleftrightarrow-\mathrm{sup}\mathrm{RHom}_R(M,N)_\mathfrak{p}=
\mathrm{depth}_{R_\mathfrak{p}}N_\mathfrak{p}
+\mathrm{inf}M_\mathfrak{p}\\
&\Longleftrightarrow \mathrm{depth}_{R_\mathfrak{p}}N_\mathfrak{p}=-\mathrm{sup}N_\mathfrak{p}\ \textrm{and}\
M_\mathfrak{p}\not\simeq0.\end{aligned}$\end{center}Hence we obtain the desired equality.

(2) Using (1) and the isomorphism $D_R(M\otimes^\mathrm{L}_RN)\simeq\mathrm{RHom}_R(M,D_R(N))$ in $\mathrm{D}(R)$.
\end{proof}

\begin{cor}\label{lem:5.17}{\it{$\mathrm{(1)}$ Let $M$ be in $\mathrm{D}_-(R)$. One has that $\mathrm{Ass}_RM\cap\mathrm{Max}R\neq\emptyset$ if and only if $\mathrm{RHom}_R(R/\mathfrak{m},M)\not\simeq0$.

$\mathrm{(2)}$ Let $M$ be in $\mathrm{D}_+(R)$. Then $\mathrm{Coass}_RM\cap\mathrm{Max}R\neq\emptyset$ if and only if $R/\mathfrak{m}\otimes^\mathrm{L}_RM\not\simeq0$.}}
\end{cor}

The following result is an extension of Nakayama lemma.

\begin{prop}\label{lem:5.10}{\it{Let $\mathfrak{a}$ be an ideal of $R$ such that $\mathfrak{a}\subseteq J(R)$ the Jacobson radical of $R$.

 $\mathrm{(1)}$ If $M$ is in $\mathrm{D}_-(R)$ such that $\mathrm{Ass}_RM\cap\mathrm{Max}R\neq\emptyset$, then $\mathrm{RHom}_R(R/\mathfrak{a},M)\not\simeq0$.

$\mathrm{(2)}$ If $M$ is in $\mathrm{D}_+(R)$ such that $\mathrm{Coass}_RM\cap\mathrm{Max}R\neq\emptyset$, then $R/\mathfrak{a}\otimes^\mathrm{L}_RM\not\simeq0$.}}
\end{prop}
\begin{proof} (1) Given $\mathfrak{m}\in\mathrm{Ass}_RM\cap\mathrm{Max}R$ and set $s=\mathrm{sup}M_\mathfrak{m}$. Then $\mathrm{H}_s(\mathrm{RHom}_{R_\mathfrak{m}}(k(\mathfrak{m}),M_\mathfrak{m}))\cong
\mathrm{Hom}_{R_\mathfrak{m}}(k(\mathfrak{m}),\mathrm{H}_s(M_\mathfrak{m}))\neq0$.
So $\mathrm{H}_s(\mathrm{RHom}_{R_\mathfrak{m}}((R/\mathfrak{a})_\mathfrak{m},M_\mathfrak{m}))
\cong\mathrm{Hom}_{R_\mathfrak{m}}((R/\mathfrak{a})_\mathfrak{m},\mathrm{H}_s(M_\mathfrak{m}))\neq0$ since the map $(R/\mathfrak{a})_\mathfrak{m}\twoheadrightarrow (R/\mathfrak{m})_\mathfrak{m}$ is surjective, which implies that $\mathrm{RHom}_R(R/\mathfrak{a},M)\not\simeq0$.

(2) By Theorem \ref{lem:5.5}, $\mathrm{Ass}_RD_R(M)\cap\mathrm{Max}R\neq\emptyset$. Hence $\mathrm{RHom}_R(R/\mathfrak{a},D_R(M))\not\simeq0$ by (1), which implies that $R/\mathfrak{a}\otimes^\mathrm{L}_RM\not\simeq0$.
\end{proof}

\bigskip \centerline {\bf ACKNOWLEDGEMENTS}  I am very grateful to Peder Thompson for helpful conversations. This research was partially supported by National Natural Science Foundation of China (11761060,11901463), Improvement of Young Teachers$'$ Scientific Research Ability (NWNU-LKQN-18-30) and Innovation Ability Enhancement Project of Gansu Higher Education Institutions (2019A-002).

\end{document}